\documentclass[11pt]{article}
\usepackage[T1]{fontenc}
\usepackage{lmodern}
\usepackage[margin=1in]{geometry}
\usepackage{amsmath,amssymb,amsthm,mathtools}
\usepackage{microtype}
\usepackage{etoolbox}
\usepackage[numbers,sort&compress]{natbib}
\usepackage[hidelinks]{hyperref}
\usepackage{fancyhdr}
\hypersetup{pdftitle={A solution to the Erdos Problem \#1040},
  pdfauthor={Ioannis Tzachristas},
  pdfsubject={The capacity-one case of Erdos problem 1040},
  pdfkeywords={logarithmic capacity, polynomial lemniscates, balayage, harmonic measure}}
\newcommand{\C}{\mathbb C}

\newcommand{\D}{\mathbb D}
\DeclareMathOperator{\caplog}{cap}
\DeclareMathOperator{\supp}{supp}
\DeclareMathOperator{\dist}{dist}
\DeclareMathOperator{\Rea}{Re}
\newcommand{\Prob}{\mathcal P}
\newcommand{\qe}{\text{q.e.}}
\newcommand{\areaAE}{\text{a.e.}}
\newtheorem{theorem}{Theorem}[section]
\newtheorem{lemma}[theorem]{Lemma}
\newtheorem{proposition}[theorem]{Proposition}
\newtheorem{corollary}[theorem]{Corollary}
\theoremstyle{remark}
\newtheorem{remark}[theorem]{Remark}
\numberwithin{equation}{section}
\AtBeginEnvironment{thebibliography}{\raggedright}

\title{A solution to the Erd\H{o}s Problem \#1040}
\author{Ioannis Tzachristas}
\date{5 September 2026}

\begin{document}
\maketitle

\begin{abstract}
For a compact set $K\subset\C$, let $\vartheta(K)$ be the infimum of
the planar areas of the unit lemniscates of all monic polynomials with
zeros in $K$, allowing arbitrary degree and repeated zeros.
We prove that $\vartheta(K)=0$ whenever
$\caplog(K)=1$, with no regularity assumption on $K$.
The proof uses a centered harmonic polynomial that is positive on all
but a set of arbitrarily small area in the polynomial hull of $K$.
A Fourier average of exterior harmonic measures realizes this polynomial
as the logarithmic potential of a signed measure having bounded density
with respect to the equilibrium measure. A positive perturbation and
an $L^1$ approximation by empirical measures then produce the required
polynomials. This extends the smooth-boundary result of Krishnapur,
Lundberg, and Ramachandran to arbitrary compact sets of capacity one.
Together with the capacity-greater-than-one theorem of Ghosh and
Ramachandran and an elementary argument for unbounded sets, it follows
that $\vartheta(F)=0$ for every closed infinite set $F\subset\C$
of transfinite diameter at least one, answering the vanishing question
in Erd\H{o}s Problem~1040.
\end{abstract}


\section{Introduction and statement of the result}

Write $m$ for planar Lebesgue measure. For a nonempty set $F\subset\C$,
define
\begin{equation}\label{eq:def}
 \vartheta(F)=\inf_{n\ge1}\ \inf_{z_1,\ldots,z_n\in F}
 m\left(\left\{z\in\C:
       \left|\prod_{j=1}^n(z-z_j)\right|<1\right\}\right).
\end{equation}
The points $z_j$ need not be distinct. In particular, every polynomial
in this definition is monic and has all its zeros in the prescribed set.

Erd\H{o}s, Herzog, and Piranian asked whether this infimum is determined
by transfinite diameter, and in particular whether it vanishes for every
closed infinite set of transfinite diameter at least one
\cite[Problem~4, p.~135]{EHP1958}. For compact sets, transfinite diameter
equals logarithmic capacity; we use the notation $\caplog$ for either
quantity. These questions are recorded as Erd\H{o}s problem~1040.
They have distinct answers. A counterexample to determination by
capacity alone is reported by Feng et al.\ \cite[Section~3.3]{Feng2026}.
Ghosh and Ramachandran \cite[Example~2.1]{GR2026} give compact sets
with different values of $\vartheta$ at every prescribed capacity in
$(0,1)$. The present paper addresses the vanishing question at and
above capacity one.

Erd\H{o}s, Herzog, and Piranian already proved vanishing when the
allowed zeros lie on the unit circle
\cite[Theorem~4 and its corollary, p.~133]{EHP1958}. They also proved
that for a closed set of transfinite diameter less than one every
admissible unit lemniscate contains a disk of a fixed positive radius
\cite[Theorem~6, p.~135]{EHP1958}. In particular, the threshold one
in the vanishing question cannot be lowered.

The capacity-greater-than-one case was proved without regularity
assumptions by Ghosh and Ramachandran \cite[Theorem~3.1]{GR2026}.
They also establish the sharp exponential decay rate
\cite[Theorem~3.3]{GR2026}. At capacity one, Krishnapur, Lundberg, and
Ramachandran prove the vanishing assertion when $K$ is the closure of a
bounded open set with $C^2$ boundary \cite[Theorem~6]{KLR2025}.
The introduction of \cite[version~3, 27 August 2026]{GR2026}
explicitly leaves general compact sets of capacity one open.
We prove the remaining compact case without any regularity assumption.

\begin{theorem}\label{thm:main}
Let $K\subset\C$ be compact with $\caplog(K)=1$. For every $a>0$
there are an integer $n\ge1$ and points $z_1,\ldots,z_n\in K$ such that
the monic polynomial $p(z)=\prod_{j=1}^n(z-z_j)$ satisfies
\[
 m\{z\in\C:|p(z)|<1\}<a.
\]
Equivalently, $\vartheta(K)=0$.
\end{theorem}

The proof follows the broad strategy of \cite[Section~5]{KLR2025}:
find a centered harmonic function positive on most of the relevant set,
represent it by a controlled signed measure, perturb the equilibrium
measure, and discretize. To handle arbitrary compact sets, we use
three ingredients.
First, a Cauchy-transform singularity and Hahn--Banach argument give a
harmonic polynomial with the required positivity outside a thin corridor.
Second, averaging exterior harmonic measures realizes this polynomial
by a measure whose density is bounded relative to the equilibrium
measure, without assumptions on the boundary. Third, convergence of
logarithmic potentials in planar $L^1$ replaces geometric discretization
of a smooth boundary. All three ingredients are proved below.

The construction is existential. It gives no effective degree bound,
explicit sequence of zeros, or uniform rate of decay in the degree.
Section~\ref{sec:other} combines Theorem~\ref{thm:main} with the
capacity-greater-than-one case and the unbounded case to prove
Corollary~\ref{cor:closed}: every closed infinite set of transfinite
diameter at least one has $\vartheta(F)=0$.

\section{Potential theory and harmonic-measure domination}
\label{sec:prelim}

Write $\Prob(K)$ for the Borel probability measures supported on $K$.
For a compactly supported finite positive measure $\rho$, use the
positive-log convention
\begin{equation}\label{eq:potential}
 U_\rho(z)=\int\log|z-t|\,d\rho(t).
\end{equation}
The same notation is used for a finite signed measure wherever its
positive and negative parts have finite potentials, and elsewhere as
an almost-everywhere defined locally integrable function.
On any bounded planar disk $B$, for $t$ in a fixed compact set $K$,
\begin{equation}\label{eq:log-int}
 \sup_{t\in K}\int_B |\log|z-t||\,dm(z)<\infty.
\end{equation}
Indeed, after translating by $t$, the integration region lies in one
fixed disk, and the logarithmic singularity is integrable there.
Consequently all the potentials used below belong to $L^1(B)$.

Let $K$ be compact and nonpolar, let $H=\widehat K$ be its polynomial
hull, and put $\Omega=\C\setminus H$. Thus $H$ is $K$ together with
the bounded components of $\C\setminus K$; it is compact and full,
meaning that its complement is connected. When discussing harmonic
measure at infinity, we adjoin infinity to $\Omega$ and regard it as
a domain on the Riemann sphere. Write $\mu$ for the equilibrium
probability measure of $K$.

We recall the following standard facts from logarithmic potential
theory; see \cite{Ransford1995,SaffTotik1997} and
\cite[Theorems~1.9 and~1.12, Definition~3.4]{Saff2010}.
One has $\caplog(H)=\caplog(K)$, the equilibrium measures of $H$
and $K$ coincide, and $\supp\mu\subset\partial H\subset K$.
Frostman's theorem and the exterior Green-function identity, with
our sign convention, give
\begin{align}
 U_\mu(z)&=\log\caplog(K) &&\qe\text{ on }H,\label{eq:frostman}\\
 U_\mu(w)&=\log\caplog(K)+g_\Omega(w,\infty)
       &&w\in\Omega.\label{eq:green}
\end{align}
Here $g_\Omega(w,\infty)>0$ for finite $w\in\Omega$.
The abbreviation ``q.e.'' means outside a set of logarithmic capacity
zero. Such sets have planar area zero, so these identities also hold
almost everywhere for $m$.

For finite $w\in\Omega$, let $\nu_w=\omega_\Omega(w,\cdot)$ be
harmonic measure on $\partial H$. It is a probability measure, and
harmonic measure at infinity is $\mu$. Exterior logarithmic balayage
gives
\begin{equation}\label{eq:balayage}
 U_{\nu_w}(z)=\log|z-w|-g_\Omega(w,\infty)
 \qquad\qe\text{ on }H.
\end{equation}
This formula requires no boundary smoothness and holds throughout
$H$, including its interior. In the negative-log convention,
\cite[eqs.~(3.15)--(3.16), pp.~188--189]{Saff2010} states the formula
with the opposite sign for the additive constant. The identification
with harmonic measure is also stated in
\cite[Section~2, eq.~(2.4)]{Tookos2005}.
For capacity one, \eqref{eq:balayage} becomes
\begin{equation}\label{eq:balayage-one}
 U_{\nu_w}(z)=\log|z-w|-U_\mu(w)
 \qquad\qe\text{ on }H.
\end{equation}

We record explicitly the comparison needed to keep the subsequent
perturbation positive.

\begin{lemma}\label{lem:harnack}
Choose $0<r<R$ with $H\subset\{z:|z|<r\}$. For every $|w|=R$,
\begin{equation}\label{eq:domination}
 \nu_w\le C_{r,R}\mu,\qquad C_{r,R}=\frac{R+r}{R-r}.
\end{equation}
In particular, $R=2r$ gives $\nu_w\le3\mu$.
\end{lemma}

\begin{proof}
For a Borel set $A\subset\partial H$, the function
$u_A(w)=\nu_w(A)$ is bounded, nonnegative, and harmonic in $\Omega$,
with value $\mu(A)$ at infinity. Thus
$v_A(\zeta)=u_A(r/\zeta)$ extends harmonically across $\zeta=0$
to the unit disk, with $v_A(0)=\mu(A)$.
The disk Harnack inequality gives
\[
 v_A(\zeta)\le\frac{1+|\zeta|}{1-|\zeta|}\,v_A(0).
\]
Apply it with $\zeta=r/w$ and $|w|=R$.
The resulting inequality holds for every Borel $A$, which is exactly
the measure inequality \eqref{eq:domination}. The zero function is
included by continuity, or by applying Harnack to $v_A+\epsilon$
and letting $\epsilon\downarrow0$.
\end{proof}

\section{Harmonic polynomials positive outside a thin corridor}
\label{sec:corridor}

The next lemma is independent of potential theory. Its probability
measure need not be an equilibrium measure.

\begin{lemma}\label{lem:corridor}
Let $H\subset\C$ be compact and full, with $m(H)>0$, and let $\mu$
be a probability measure supported on $H$. Given $\eta>0$, there
exist a nonempty compact set $E\subset H$ and a complex polynomial
$P$ such that
\begin{equation}\label{eq:corridor}
 m(H\setminus E)<\eta,\qquad
 \int P\,d\mu=1,\qquad \sup_E|P|<\tfrac12.
\end{equation}
Consequently $h=\Rea(1-P)$ is a real harmonic polynomial satisfying
\begin{equation}\label{eq:centered}
 \int h\,d\mu=0,\qquad h>\tfrac12\text{ on }E.
\end{equation}
\end{lemma}

\begin{proof}
Consider the Cauchy transform
\[
 C_\mu(w)=\int\frac{d\mu(z)}{w-z},\qquad w\in\C\setminus H.
\]
It is holomorphic there and satisfies
$C_\mu(w)=w^{-1}+O(w^{-2})$ near infinity. Fix $w_0$ outside a
large disk containing $H$, and consider the Taylor series of $C_\mu$
at $w_0$.

Its radius of convergence $R_0$ is finite and positive. Otherwise its
sum would be an entire function agreeing with $C_\mu$ on the connected
set $\C\setminus H$, by the identity theorem. This entire function
would tend to zero at infinity and hence would be identically zero by
Liouville's theorem, contradicting the coefficient of $w^{-1}$.

Let $T$ denote the Taylor sum on the disk
$D=B(w_0,R_0)$. There is a point $a\in\partial D$ at which $T$
does not extend holomorphically across the boundary. This is the usual
boundary-singularity property of a Taylor disk with finite radius of
convergence: extensions at every boundary point would, by a finite
cover of the boundary circle, extend the Taylor sum to a larger disk.

Set $\gamma=[w_0,a]$. For $\delta>0$, let
\[
 V_\delta=\{z:\dist(z,\gamma)<\delta\},\qquad
 E=H\setminus V_\delta.
\]
Since a line segment has planar area zero, choose $\delta>0$ such that
\[
 m(H\setminus E)<\min\{\eta,m(H)/2\}.
\]
Then $E$ is compact and has positive area.

Define the complex-linear functional
$\Lambda(Q)=\int Q\,d\mu$ on complex polynomials. We claim that
there is no finite $M$ such that
\begin{equation}\label{eq:bounded-functional}
 |\Lambda(Q)|\le M\sup_E|Q|\quad\text{for every polynomial }Q.
\end{equation}
Suppose otherwise. Then $\Lambda$ is a bounded functional on the
polynomial restrictions in $C(E)$. Complex Hahn--Banach and the Riesz
representation theorem give a finite complex Borel measure $\lambda$
supported on $E$ such that
\[
 \int Q\,d\lambda=\int Q\,d\mu
 \quad\text{for all complex polynomials }Q.
\]
Their analytic moments coincide, so their Cauchy transforms have the
same Laurent expansion near infinity. Hence $C_\lambda=C_\mu$
near infinity and, by the identity theorem, throughout $\C\setminus H$.
In particular they agree near $w_0$.

Since $E$ is disjoint from $V_\delta$, the function $C_\lambda$
is holomorphic on $V_\delta$. Both $D$ and $V_\delta$ are convex,
so $D\cap V_\delta$ is connected. The identity theorem now gives
$C_\lambda=T$ on this intersection. But $V_\delta$ contains an open
neighborhood of $a$. Thus $C_\lambda$ extends $T$ across $a$, a
contradiction. This proves the claim.

Choose $Q$ for which $|\Lambda(Q)|>2\sup_E|Q|$ and set
$P=Q/\Lambda(Q)$. This gives \eqref{eq:corridor}, and
\eqref{eq:centered} follows by taking real parts.
\end{proof}

\begin{remark}
The removed corridor is allowed to pass through $H$. There is no
assertion that $a\in H$, and the proof does not require $E$ to be full
or polynomially convex. These distinctions avoid making any assumption
about continuation of $C_\mu$ through the original support.
\end{remark}

\section{Realization by a controlled signed measure}
\label{sec:realization}

Return to a compact $K$ with $\caplog(K)=1$, its hull $H$, and its
equilibrium measure $\mu$.

\begin{lemma}\label{lem:realization}
Let $h$ be a real harmonic polynomial with $\int h\,d\mu=0$.
There is a finite real signed measure $\sigma$ supported on $K$ such that
\begin{equation}\label{eq:realization}
 \sigma(K)=0,\qquad |\sigma|\le C_h\mu,\qquad
 U_\sigma=h\quad m\text{-}\areaAE\text{ on }H
\end{equation}
for some finite constant $C_h$.
\end{lemma}

\begin{proof}
Write $m_k=\int z^k\,d\mu(z)$. The centered condition allows us to
write
\begin{equation}\label{eq:h-expansion}
 h(z)=\Rea\sum_{k=1}^d a_k(z^k-m_k).
\end{equation}
If $h=0$, take $\sigma=0$. Otherwise choose $R$ larger than the
moduli of all points of $H$, and put $w_\theta=Re^{i\theta}$.
The uniformly convergent logarithmic expansion and
\eqref{eq:balayage-one} give, for each $\theta$,
\begin{equation}\label{eq:swept-expansion}
 U_{\nu_{w_\theta}}(z)
 =-\Rea\sum_{j=1}^{\infty}
       \frac{z^j-m_j}{jR^j}e^{-ij\theta}
 \quad m\text{-}\areaAE\text{ on }H.
\end{equation}
The series on the right converges uniformly for $(z,\theta)\in
H\times[0,2\pi]$.

Define the real trigonometric polynomial
\begin{equation}\label{eq:fourier-weight}
 b(\theta)=-2\sum_{k=1}^d kR^k\Rea(a_ke^{ik\theta}),
\end{equation}
and the real signed measure
\begin{equation}\label{eq:sigma}
 \sigma(A)=\int_0^{2\pi} b(\theta)\nu_{w_\theta}(A)
                          \frac{d\theta}{2\pi}.
\end{equation}
Harmonic measure is a Borel probability kernel in its pole, so this
integral defines a countably additive finite signed measure. For
example, countable additivity follows from dominated convergence
applied to disjoint unions, since $b$ is integrable and each
$\nu_{w_\theta}$ is a probability. Its total mass is zero because
$b$ has zero mean. Its support lies in $\partial H\subset K$.

By Lemma~\ref{lem:harnack}, for a constant $C_R$ independent of
$\theta$, one has $\nu_{w_\theta}\le C_R\mu$. Therefore
\begin{equation}\label{eq:variation-bound}
 |\sigma|(A)\le C_R\left(\int_0^{2\pi}|b(\theta)|
                       \frac{d\theta}{2\pi}\right)\mu(A).
\end{equation}
To see the total-variation inequality, first bound $|\sigma(A_j)|$
for each set in a finite Borel partition of $A$, sum, and then take
the supremum over such partitions.

We may integrate \eqref{eq:swept-expansion} against
$b(\theta)d\theta/(2\pi)$ and interchange the logarithmic integrations.
Indeed, for any bounded disk $B$ containing $H$,
\[
 \int_0^{2\pi}\!|b(\theta)|
 \int_B\!\int |\log|z-t||\,d\nu_{w_\theta}(t)\,dm(z)
 \frac{d\theta}{2\pi}<\infty
\]
by domination and \eqref{eq:log-int}. Fubini also deals with the
exceptional sets in \eqref{eq:swept-expansion}: they may depend on
$\theta$, but each has planar area zero. No common exceptional set
for all poles is required.

For $j,k\ge1$, Fourier orthogonality gives
\[
 \int_0^{2\pi}\Rea(a_ke^{ik\theta})
            \Rea(v_je^{-ij\theta})\frac{d\theta}{2\pi}
 =\begin{cases}
   \tfrac12\Rea(a_kv_k),&j=k,\\
   0,&j\ne k.
  \end{cases}
\]
Equations \eqref{eq:swept-expansion}--\eqref{eq:sigma} consequently
yield $U_\sigma=h$ almost everywhere on $H$. Together with
\eqref{eq:variation-bound}, this proves \eqref{eq:realization}.
\end{proof}

\section{From probability measures to monic polynomials}
\label{sec:discretization}

The next lemma uses planar area, rather than a measure on the boundary.
It is valid for every compact set, including sets with irregular or
positive-area boundary.

\begin{lemma}\label{lem:L1}
Let $K\subset\C$ be compact, and let $\rho_j,\rho\in\Prob(K)$
be probability measures with $\rho_j\to\rho$ weakly. For every
bounded disk $B\subset\C$,
\[
 \|U_{\rho_j}-U_\rho\|_{L^1(B)}\longrightarrow0.
\]
\end{lemma}

\begin{proof}
For $L>0$, set
$k_L(z,t)=\max\{\log|z-t|,-L\}$, taking the value $-L$ at $z=t$.
For fixed $L$, this is continuous and bounded on $\overline B\times K$.
Weak convergence gives pointwise convergence of the corresponding
truncated potentials, and bounded convergence gives their convergence
in $L^1(B)$.

For any probability measure $\tau$ on $K$, Tonelli's theorem gives
the uniform truncation estimate
\begin{align}
 &\int_B\!\int
       \bigl(k_L(z,t)-\log|z-t|\bigr)\,d\tau(t)\,dm(z)\notag\\
 &\hspace{1cm}\le
   \int_{|u|<e^{-L}}(-L-\log|u|)\,dm(u)
   =\frac{\pi}{2}e^{-2L}.\label{eq:tail}
\end{align}
The triangle inequality, followed first by $j\to\infty$ and then
by $L\to\infty$, proves the assertion.
\end{proof}

\begin{proposition}\label{prop:discrete}
For compact $K\subset\C$ and $\rho\in\Prob(K)$,
\begin{equation}\label{eq:discrete-area}
 \vartheta(K)\le m\{z\in\C:U_\rho(z)\le0\}.
\end{equation}
\end{proposition}

\begin{proof}
Choose equally weighted empirical probability measures
\[
 \rho_N=\frac1N\sum_{j=1}^N\delta_{z_{j,N}},\qquad z_{j,N}\in K,
\]
converging weakly to $\rho$, with $N\to\infty$ along a sequence.
For completeness, partition $K$ into finitely many Borel sets of
diameter tending to zero, approximate their masses by nonnegative
rational numbers with a common denominator and total mass one, and
place the prescribed multiplicities at a point of each nonempty
partition cell. Taking the total error in the masses to zero produces
the stated approximation. Increasing denominators if necessary makes
$N\to\infty$.

Set $p_N(z)=\prod_{j=1}^N(z-z_{j,N})$. Then
$N^{-1}\log|p_N|=U_{\rho_N}$, with the usual value $-\infty$ at
zeros. If $K\subset\{|t|\le r\}$, every probability measure
$\tau$ on $K$ satisfies
\begin{equation}\label{eq:bounding-disk}
 U_\tau(z)>0\qquad(|z|>r+1).
\end{equation}
Choose one bounded disk $B$ containing $\{|z|\le r+1\}$.
It contains all the unit lemniscates under consideration.

For $s>0$, the part of $\{U_{\rho_N}<0\}$ on which $U_\rho>s$
has area at most $s^{-1}\|U_{\rho_N}-U_\rho\|_{L^1(B)}$.
Thus
\[
 m\{U_{\rho_N}<0\}
 \le m\bigl(B\cap\{U_\rho\le s\}\bigr)
       +s^{-1}\|U_{\rho_N}-U_\rho\|_{L^1(B)}.
\]
By Lemma~\ref{lem:L1}, first letting $N\to\infty$ and then
$s\downarrow0$ yields
\[
 \limsup_N m\{|p_N|<1\}
 \le m\bigl(B\cap\{U_\rho\le0\}\bigr)
 =m\{U_\rho\le0\}.
\]
This proves \eqref{eq:discrete-area}.
\end{proof}

\begin{remark}
The limiting set in \eqref{eq:discrete-area} is the nonpositive
sublevel set. No assertion about the area of the zero level set is
needed. This distinction is useful because the equilibrium potential
at capacity one can vanish on a set of positive area.
\end{remark}

\section{Proof of the capacity-one theorem}
\label{sec:proof}

\begin{proof}[Proof of Theorem~\ref{thm:main}]
Let $H=\widehat K$ and let $\mu$ be the equilibrium measure of $K$.
If $m(H)=0$, then $U_\mu>0$ outside $H$, so
$m\{U_\mu\le0\}=0$. Proposition~\ref{prop:discrete} gives
$\vartheta(K)=0$ immediately.

Suppose $m(H)>0$ and fix $\eta>0$.
Lemma~\ref{lem:corridor} gives $E\subset H$ and a real harmonic
polynomial $h$ such that
\[
 m(H\setminus E)<\eta,\qquad
 \int h\,d\mu=0,\qquad h>\tfrac12\text{ on }E.
\]
By Lemma~\ref{lem:realization}, choose $\sigma$ with
$\sigma(K)=0$, $|\sigma|\le C_h\mu$, and $U_\sigma=h$
almost everywhere on $H$. For
$0<t<[2(1+C_h)]^{-1}$, define
\begin{equation}\label{eq:perturb}
 \rho_t=\mu+t\sigma.
\end{equation}
The measure $\rho_t$ is a probability supported on $K$: its density
relative to $\mu$ is at least $1-tC_h>1/2$, and its mass is one.
Since $U_\mu=0$ almost everywhere on $H$,
\[
 U_{\rho_t}=t h>t/2\quad m\text{-}\areaAE\text{ on }E.
\]
Consequently
\begin{equation}\label{eq:inside}
 m\bigl(H\cap\{U_{\rho_t}\le0\}\bigr)<\eta.
\end{equation}

For each finite $z\in\Omega=\C\setminus H$, the potential
$U_\sigma(z)$ is finite and
$U_\mu(z)=g_\Omega(z,\infty)>0$. Hence
\[
 U_{\rho_t}(z)=U_\mu(z)+tU_\sigma(z)
 \longrightarrow U_\mu(z)>0\qquad(t\downarrow0).
\]
Every nonpositive sublevel set in this family lies in one bounded disk
by \eqref{eq:bounding-disk}. Dominated convergence applied to its
intersection with $\Omega$ therefore gives
\begin{equation}\label{eq:outside}
 m\bigl(\Omega\cap\{U_{\rho_t}\le0\}\bigr)
 \longrightarrow0\qquad(t\downarrow0).
\end{equation}
Choose $t$ small enough that the left-hand side is less than $\eta$.
Combining \eqref{eq:inside} and \eqref{eq:outside}, we obtain
$m\{U_{\rho_t}\le0\}<2\eta$.

Proposition~\ref{prop:discrete} now gives $\vartheta(K)<2\eta$.
Since $\eta>0$ was arbitrary and $\vartheta(K)\ge0$, the theorem
follows. Equivalently, given $a>0$, first take $2\eta<a$ and then
choose an empirical approximation far enough along the sequence to
obtain a polynomial with unit-lemniscate area less than $a$.
\end{proof}

\section{Completion of the vanishing question}
\label{sec:other}

For completeness, we first recover the qualitative
capacity-greater-than-one theorem of Ghosh and Ramachandran
\cite[Theorem~3.1]{GR2026} from Proposition~\ref{prop:discrete}.

\begin{proposition}\label{prop:greater}
If $K\subset\C$ is compact and $\caplog(K)>1$, then
$\vartheta(K)=0$.
\end{proposition}

\begin{proof}
The equilibrium potential equals $\log\caplog(K)>0$ quasi-everywhere
on its hull and equals this constant plus the positive exterior Green
function outside the hull. Its nonpositive set therefore has planar
area zero. Apply Proposition~\ref{prop:discrete}.
\end{proof}

The sharp exponential decay rate established in
\cite[Theorem~3.3]{GR2026} is a stronger conclusion; the argument here
recovers only the vanishing of the infimum.

\begin{proposition}\label{prop:unbounded}
If $F\subset\C$ is unbounded, then $\vartheta(F)=0$.
\end{proposition}

\begin{proof}
Choose $a,b\in F$ at distance $D>2$ and set $p(z)=(z-a)(z-b)$.
If $|p(z)|<1$, the smaller of $|z-a|$ and $|z-b|$ is less than one.
The larger is therefore greater than $D-1$, so the smaller is less
than $(D-1)^{-1}$. Thus
\[
 \{|p|<1\}\subset
 B\bigl(a,(D-1)^{-1}\bigr)\cup B\bigl(b,(D-1)^{-1}\bigr),
\]
and $m\{|p|<1\}\le2\pi(D-1)^{-2}$. Let $D\to\infty$.
\end{proof}

\begin{corollary}\label{cor:closed}
For every closed infinite set $F\subset\C$ of transfinite diameter
at least one, $\vartheta(F)=0$.
\end{corollary}

\begin{proof}
If $F$ is bounded, it is compact; apply Theorem~\ref{thm:main} or
Proposition~\ref{prop:greater}. If it is unbounded, use
Proposition~\ref{prop:unbounded}.
\end{proof}

The unbounded argument does not assume the existence of a compact
subset of capacity greater than one. Such an assumption would fail
for an unbounded closed discrete set.

\section{Further remarks}

Theorem~\ref{thm:main} removes the boundary regularity assumption
from the capacity-one theorem of
\cite[Theorem~6]{KLR2025}. The general perturbation strategy comes
from that work, while Sections~\ref{sec:corridor}--\ref{sec:discretization}
provide the arguments needed for arbitrary compact sets. In particular,
the proof controls planar area through $L^1$ convergence of potentials;
it does not require uniform convergence at their logarithmic
singularities.

The order of choices is essential: first the corridor and harmonic
polynomial, then the signed measure, then the positive perturbation
size, and finally the empirical approximation. None of the intermediate
constants is asserted to be uniform as the corridor narrows. Obtaining
degree bounds or explicit root configurations is a separate question.

Thus Corollary~\ref{cor:closed} gives the affirmative answer to the
vanishing question in \cite[Problem~4]{EHP1958}, while the separate
question of determination by transfinite diameter alone has the
negative answer given by the counterexamples in
\cite[Section~3.3]{Feng2026} and \cite[Example~2.1]{GR2026}.

\appendix
\section{A normalization check on the unit circle}

Let $K=\{z:|z|=1\}$, so $H=\overline\D$ and
$d\mu(e^{i\phi})=d\phi/(2\pi)$. For the centered harmonic
polynomial $h(z)=\Rea z$, define
\[
 d\sigma(e^{i\phi})=-2\cos\phi\,\frac{d\phi}{2\pi}.
\]
Then $\sigma(K)=0$ and $|\sigma|\le2\mu$.
For $|z|<1$, expand the logarithm and integrate term by term:
\[
 \int\log|z-e^{i\phi}|\,d\sigma(e^{i\phi})
 =\int\left(-\Rea\sum_{k\ge1}\frac{z^ke^{-ik\phi}}{k}\right)
          (-2\cos\phi)\frac{d\phi}{2\pi}
 =\Rea z.
\]
For $|z|>1$, the exterior expansion similarly gives
$U_\sigma(z)=\Rea(1/z)$. Continuity of the logarithmic single-layer
potential with this smooth density extends the interior equality to
$|z|=1$. In particular, for $0<t<1/2$, the measure
\[
 d\rho_t(e^{i\phi})=(1-2t\cos\phi)\frac{d\phi}{2\pi}
\]
is positive and has potential $t\Rea z$ on the closed disk.
This example checks the sign and factor in the representation lemma;
by itself it does not give arbitrarily small area. That step requires
the harmonic polynomials furnished by Lemma~\ref{lem:corridor}.

\clearpage
\section*{AI assistance and review status}

The AI-assisted proposed result was obtained through the author's
prompting strategy using OpenAI's
GPT-6-Astra model through Codex. AI assistance included proof
exploration, mathematical reasoning, literature checks, manuscript
drafting, and internal reviews by parallel reasoning agents.
These internal checks are not independent human verification.
No proof-assistant formalization or external mathematical endorsement
is asserted in this draft. The bibliography distinguishes the existing
results and method from the proposed extension.

\bibliographystyle{plainnat}
\bibliography{references}
\end{document}